\documentclass[11pt]{amsart}

\usepackage[utf8]{inputenc}
\usepackage{amsmath}
\usepackage{amsthm}
\usepackage{amssymb}
\usepackage{color}
\usepackage{hyperref}
\usepackage{stackrel}
\usepackage{ulem}

\usepackage{amsrefs}
\usepackage{xcolor}
\usepackage{enumerate}
\usepackage{verbatim}
\usepackage{cancel}
\usepackage{setspace}
\usepackage{mathtools}
\usepackage{comment}

\newtheorem{theorem}{Theorem}[section]
\newtheorem{lemma}[theorem]{Lemma}
\newtheorem{example}[theorem]{Example}
\newtheorem{definition}[theorem]{Definition}

\newtheorem{proposition}[theorem]{Proposition}
\newtheorem{remark}[theorem]{Remark}
\newtheorem{corollary}[theorem]{Corollary}

\DeclareMathOperator{\supp}{supp}
\DeclareMathOperator{\dist}{dist}
\DeclareMathOperator{\Hess}{Hess}
\DeclareMathOperator{\spann}{span}

\DeclareMathOperator{\epi}{epi}

\DeclareMathOperator{\x}{\boldsymbol{x}}
\DeclareMathOperator{\y}{\boldsymbol{y}}

\begin{document}
	
	\title[Mollifier smoothings of strongly convex $C^0$-Finsler structures]{Mollifier smoothings of strongly convex $C^0$-Finsler structures}
	
	\author{Ryuichi Fukuoka}
	\address{Department of Mathematics, State University of Maring\'a, Jd. Universitário, Maring\'a, 87020-900, Paraná, Brazil \\ email: rfukuoka@uem.br}
	
	\author{Anderson Macedo Setti}
	\address{Department of Exact and Earth Sciences, Federal University of Mato Grosso, Boa Esperança, Cuiabá, 78060-900, Mato Grosso, Brazil, email: anderson.setti@ufmt.br}
	
	\date{\today}

\begin{abstract}
	Let $M$ be a smooth manifold and $TM$ its tangent bundle. 
	A $C^0$-Finsler structure of $M$ is a continuous function $F: TM \to [0, \infty)$ such that restricted to each tangent space is an asymmetric norm.
	A mollifier smoothing of $F$ is a family of Finsler structures $F_\varepsilon: TM \to [0, \infty)$, parameterized by $\varepsilon > 0$ and constructed using the standard mollifier, such that $F_\varepsilon \to F$ uniformly on compact subsets as $\varepsilon \to 0$.
	In this work, we construct two mollifier smoothings of $F$. In the first, we assume that $F: TM \to [0, \infty)$ is strongly convex, that is, that $F$ restricted to each tangent space is strongly convex.
	In the second, we assume that $M=G$ is a Lie group and that $F:TG \to [0, \infty)$ is strongly convex and left-invariant, and in this case, $F_\varepsilon$ is also left-invariant.
	In both cases, we prove that if $F$ is a Finsler structure, objects such as the fundamental tensor, the Chern, Cartan, Hashiguchi and Berwald connections, and the flag curvature of $(M, F_\varepsilon)$ converge uniformly on compact subsets to the corresponding objects of $(M, F)$ as $\varepsilon \to 0$.
\end{abstract}



	\maketitle


\section{Introduction}

Let $M$ be an $n$-dimensional smooth manifold. We denote by $T_xM$ and $T^*_xM$ the tangent and cotangent spaces at $x \in M$, respectively. The tangent and cotangent bundles of $M$ are denoted by $TM$ and $T^*M$, respectively.

A Finsler structure of $M$ is a function $F:TM \to [0, \infty)$ which is smooth on the slit tangent bundle $TM\setminus 0 := \{(x, y) \in TM : y \neq 0\}$ and such that its restriction to each tangent space $F(x, \cdot): T_xM \to [0, \infty)$ is a Minkowski norm. 
In the literature, there are two alternative definitions of Finsler structure when $F:TM \to [0, \infty)$ is a continuous function: in the second definition, $F(x,\cdot)$ is a norm (see \cite{Burago}), and in the third one $F(x,\cdot)$ is an asymmetric norm (see \cite{MatveevTroyanov}). 
The smooth version is by far the most common and will be used in this work (see Definition \ref{definição de estrutura de Finlser}). In \cite{CordovaFukuokaNeves, Fukuoka-large-family}, the term $C^0$-Finsler structure was used for the second definition, whereas, in \cite{Fukuoka-Setti, Fukuoka-Setti2, PrudencioFukuoka, Fukuoka-Rodrigues}, the same term was used for the third definition. This latter definition is a natural generalization of (smooth) Finsler structure, and it will be used in this work (see Definition \ref{define estutura de Finsler C0}).


Consider a locally finite  smooth structure $\{U_\lambda \subset M, {\x_\lambda }: U_\lambda \rightarrow \mathbb R^n \}_{\lambda \in \Lambda}$ on a smooth manifold $M$ and a smooth partition of the unity $\{\varphi_\lambda\}_{\lambda \in \Lambda}$ subordinated to $\{U_\lambda\}_{\lambda \in \Lambda}$. For each $\lambda$, we have the canonical coordinate system $(\x_\lambda, \y_\lambda) = (x_\lambda^1, \ldots, x_\lambda^n, y_\lambda^1, \ldots, y_\lambda^n)$ on $TU_\lambda$.

In \cite{Fukuoka-Setti}, the authors introduce a mollifier smoothing $F_\varepsilon:TM \to [0,\infty)$ of general $C^0$-Finsler structure $F:TM \to [0,\infty)$, which is a family of Finsler structures parametrized by $\varepsilon>0$. This smoothing is done in two steps: 
Firstly, for a fixed $\lambda$, a vertical smoothing is applied on each {$F(x, \cdot):T_xU_\lambda \to [0, \infty)$} in the canonical coordinate system $(y_\lambda^1,\ldots, y_\lambda^n)$, for every $x \in U_\lambda$, what gives a family of $C^0$-Finsler structures $G_{\varepsilon,\lambda}:TU_\lambda \rightarrow [0,\infty)$ parametrized by $\varepsilon > 0$ such that their restrictions to $T_xU_\lambda$ are Minkowski norms (see Propositions 3 and 4 of \cite{Fukuoka-Setti}). 
Afterwards, a horizontal smoothing is applied on the family $G_{\varepsilon,\lambda},$ using the smooth partition of the unity, to get the mollifier smoothing $F_\varepsilon$ of $F$ (see Theorem 4 of \cite{Fukuoka-Setti}).

The construction of the vertical smoothing $G_{\varepsilon,\lambda}$ is done considering the convolution of $\zeta_\varepsilon$ and $F(x, \cdot),$ that is, $(\zeta_\varepsilon \ast_v F)(x,y)$, for every $x \in {U_\lambda},$ where $\zeta_\varepsilon$ is a modified mollifier and $\ast_v$ represents the mollifier smoothing of the vertical part.
Then ${G_{\varepsilon, \lambda}}(x, \cdot)$ is defined as the asymmetric norm whose sphere of a specified radius $r$ coincides with the level hypersurface $\{y \in T_xU_\lambda : (\zeta_\varepsilon \ast_v F)(x,y) = r\}$, for each $x \in U_\lambda$.
This construction doesn't give the desired results if $\zeta_\varepsilon$ is replaced by the standard mollifier $\eta_\varepsilon$ because $\{y \in T_xU_ \lambda : (\eta_\varepsilon \ast_v F)(x,y) = r\}$, for every $x \in U_\lambda,$ isn't necessarily the sphere of a strongly convex function, and therefore, it isn't the sphere of a Minkowski norm.

Also in \cite{Fukuoka-Setti} (see Sections 7 e 8), the authors prove that objects such as the fundamental tensor, the Chern connection and the flag curvature of $(M,F_\varepsilon)$ converge to the corresponding objects of $(M,F),$ when $F$ is a Finsler structure. 
In a coordinate system, the fundamental tensor $(g_\varepsilon)_{ij(x,y)}$ of the mollifier smoothing $F_\varepsilon$ and its partial derivatives converge to the corresponding tensor $g_{ij(x,y)}$ of $F$ and its partial derivatives respectively, which resembles the same type of convergence of the mollifier smoothing in H\"older and Sobolev spaces of functional analysis.
If $(M,F)$ is a $C^0$-Finsler manifold, it makes sense to study possible weak derivatives of $g_{ij}$, which could result in a ``weak'' version of the aforementioned geometric objects.
In fact, an example is given in Section 9 of \cite{Fukuoka-Setti}, where the authors define a two-dimensional piecewise smooth Riemannian metric, with nonzero total curvature in a set of measure zero in $M$.

\begin{remark}
	The notation used for the functions $G_{\varepsilon, \lambda}$ in \cite{Fukuoka-Setti} will be replaced by $\tilde{F}_{\varepsilon, \lambda}$, since $G$ will be used for Lie groups here.
\end{remark}

If we suppose that $F:TM \to [0,\infty)$ is a strongly convex $C^0$-Finsler structure (see Definition \ref{Definição de estrutura de Finsler de classe C^0 fortemente convexa}), the mollifier smoothing defined in \cite{Fukuoka-Setti} can be simplified using the standard mollifier $\eta_\varepsilon$ instead of $\zeta_\varepsilon$.
In other words, for strongly convex $C^0$-Finsler structures, we don't need to make any modification on the standard mollifier to have Minkowski norms on the fibers after the vertical smoothing.
Moreover, if we replace $F$ by $F^2$ during smoothing, some convergences become more elegant, since, for example, the fundamental tensor is defined using $F^2.$
Because of these observations, we consider the mollifier smoothing of $F^2(x, \cdot):{T_xU_\lambda}\to[0, \infty)$, that is, $(\eta_\varepsilon \ast_v F^2)(x,y),$ for each $x \in U_\lambda.$
The vertical smoothing $\tilde{F}_{\varepsilon,\lambda}:TU_\lambda \to [0,\infty)$ is defined in such a way that $\tilde{F}_{\varepsilon,\lambda}(x, \cdot)$ is an asymmetric norm whose sphere of a specified radius $r$ coincides with the level hypersurface $\{y \in T_xU_\lambda: (\eta_\varepsilon \ast_v F^2)(x,y) = r^2\},$ for each $x \in U_\lambda$ (see Subsection \ref{section vertical smoothing}).
The construction of the mollifier smoothing $F_\varepsilon$ of $F$ will be done using a smooth partition of unity (cf. preprint \cite{Fukuoka-Setti2}).
This replacement of $\zeta_\varepsilon$ by $\eta_\varepsilon$ will probably make calculations easier for applications. 

We now describe the organization of this paper.

In Section \ref{section Preliminaries}, we present the prerequisites necessary for this work.
Section \ref{section Mollifier smoothing of strongly convex $C^0$-Finsler structures on smooth manifolds} is divided into four subsections. In Subsections \ref{section vertical smoothing} and \ref{secao mollifier smoothing F}, we construct a mollifier smoothing of strongly convex $C^0$-Finsler structures on smooth manifolds, as mentioned above. In Subsection \ref{secao convergencia conexoes curvatura}, we prove that the fundamental tensor, Chern connection, Cartan connection, Hashiguchi connection, Berwald connection and flag curvature of the mollifier smoothing $F_\varepsilon$ converge to the corresponding objects of $F$, when $F$ is a Finsler structure. 
In the last subsection, we prove that not every strongly convex $C^0$-Finsler structure has a ``local" modulus of convexity, and, therefore, we can not expect a good behavior of the modulus of convexity under smoothings.
In Section \ref{section Mollifier smoothing of left-invariant strongly convex $C^0$-Finsler structures on Lie Groups}, we consider a Lie group $G$ with identity element $e.$ 
In Subsection \ref{subsection Mollifier smoothing of left-invariant strongly convex $C^0$-Finsler structures}, we work with left-invariant strongly convex $C^0$-Finsler structures $F$ of $G$ (see Equation \eqref{left-invariant}).
The vertical smoothing constructed in Section \ref{section vertical smoothing} is used here, and we vertically smooth $F_e:=F(e,\cdot):T_eG \to [0, \infty)$ obtaining a family of Minkowski norms $(\tilde{F}_\varepsilon)_{e}$ parameterized by $\varepsilon>0.$ 
We define the left-invariant mollifier smoothing $F_\varepsilon$ of $F$ by
\begin{align} \label{F epsilon definida em TG*}
	F_\varepsilon(x, y) = (\tilde{F}_\varepsilon)_e ((dL_{x^{-1}})_x(y))
\end{align}
for every $(x,y) \in TG,$ where $L_x:G \to G$ denotes left translation by $x \in G.$
In Section \ref{subsection Convergence of connections and curvatures on Lie groups}, we prove that the fundamental tensor, Chern connection, Cartan connection, Hashiguchi connection, Berwald connection and flag curvature of the mollifier smoothing $F_\varepsilon$ converge to the corresponding objects of $F$, when $F$ is a left-invariant Finsler structure. These convergences depend strongly of $F$ being left-invariant.

\begin{remark}
	The smoothing in \eqref{F epsilon definida em TG*} is essentially the same as that defined in Section 6 of the preprint \cite{Fukuoka-Setti2}.
	The authors present results on the modulus of convexity, for example, in Theorem 5.9: they prove that $F_e$ and $(\tilde{F}_\varepsilon)_e$ have the same modulus of convexity for small $\varepsilon$.
\end{remark}

\begin{remark}
	The notation $F_\varepsilon$, introduced in Sections 3 and 4, is used to denote different mollifier smoothings. This convention should not cause confusion.
\end{remark}

\begin{remark}
	In this work, the Einstein summation convention is in place.
\end{remark}


\section{Preliminaries}
\label{section Preliminaries}

In this section, we present the necessary prerequisites for this work. In Subsection \ref{subsection Convex functions and asymmetric norms}, we recall results related to convex functions and asymmetric norms. For instance, in Theorem \ref{Fortemente convexa - Equivalencia}, we give equivalent characterizations of strongly convex functions.
In Subsection \ref{subsection_mollifier smoothing of convex functions}, we present the theory of mollifier smoothing. Lemma \ref{Suavização de função fortemente convexa} states that mollifier smoothing of strongly convex function is strongly convex with the same modulus.
Finally, in Subsection \ref{subsection Finsler geometry}, we present objects of the Finsler geometry, such as the fundamental tensor, the Chern connection and the flag curvature.


\subsection{Convex functions and asymmetric norms}
\label{subsection Convex functions and asymmetric norms}

Throughout this section, $\mathbb{V}$ is an $n$-dimensional real vector space endowed with an inner product $\langle \cdot, \cdot \rangle.$ The norm corresponding to the inner product  $\langle \cdot, \cdot \rangle$ is denoted by $\Vert \cdot \Vert.$
For more details, see \cite{Cobzas, Rockafellar, SunYuan} for convex functions and asymmetric norms, 
\cite{Fukuoka-Setti2, Fukuoka-Rodrigues} for strongly convex functions and strongly convex asymmetric norms,  
and \cite{BaoChernShen} for Minkowski norms.


\begin{definition} \label{Fortemente convexa - definição clássica} Let $U$ be a non-empty convex subset of $\mathbb{V}$ and $f: U \rightarrow \mathbb{R}$ be a function.
	\begin{enumerate}
		\item We say that $f$ is {\rm convex} if
		\begin{align*}
			f(ty+(1-t)z) \leq tf(y)+(1-t)f(z),
		\end{align*}
		for all $t \in [0,1]$ and any $y, z \in U.$
		\item We say that $f$ is {\rm strictly convex} if
		\begin{align*}
			f(ty+(1-t)z) < tf(y)+(1-t)f(z),
		\end{align*}
		for all $t \in (0,1)$ and any $y, z \in U$ with $y \neq z.$
		\item We say that $f$ is {\rm strongly} (or {\rm uniformly}) {\rm convex with modulus} $\gamma,$ if there is a constant $\gamma>0$ such that for any $y, z \in U$ and all $t \in [0,1],$ we have
		\begin{align*}
			f(ty+(1-t)z) \leq tf(y)+(1-t)f(z)-\frac{1}{2}\gamma t(1-t) \Vert y-z \Vert^2.
		\end{align*}
	\end{enumerate}
\end{definition}


\begin{theorem} \label{Equivalencias - Fortemente convexa (Definicao Classica)} Let $U \subset \mathbb{V}$ be a non-empty open convex set and $f: U \rightarrow \mathbb{R}$ be a continuously differentiable function on $U$. Let $(y^1, \ldots, y^n)$ be a coordinate system of $\mathbb{V}$ relative to a basis of
	$\mathbb{V}.$ Then the following properties are equivalent:
	\begin{enumerate}
		\item The function $f$ is strongly convex in $U$ with modulus $\gamma > 0.$
		\item For any $y, z \in U$ 
		\begin{align*} f(z) \geq f(y) + d f_y(z-y) + \frac{1}{2}\gamma \Vert z-y\Vert^2.\end{align*}
		\item For any $y, z \in U$
		\begin{align*} (df_z - df_y) (z-y) \geq \gamma \Vert z-y\Vert^2. \end{align*}		
		\item \label{item_convexidade forte caso suave} When $f$ is twice continuously differentiable in $U$, the above properties are also equivalent to:
		
		The Hessian $\Hess f$ of $f$ is strongly (or uniformly) positive definite at every point of $U$. More precisely, 
		\begin{align*} (\Hess f)_y (v, v) := f_{y^iy^j}(y)v^i v^j \geq  \gamma \Vert v \Vert^2,  \forall y \in U, \forall v \in \mathbb{V}.\end{align*}
	\end{enumerate}
\end{theorem}
\begin{proof}
	See Chapter 1 of \cite{SunYuan}.
\end{proof}

\begin{definition}
	Consider $f: U \rightarrow \mathbb{R}$, where $U$ is a non-empty subset of $\mathbb{V}$.
	The {\rm epigraph} $\epi f$ of $f$ is a subset of $\mathbb{V} \times \mathbb{R}$ defined by
	\begin{align*}
		\epi f=\{(x,r)\in \mathbb{V} \times \mathbb{R}: \,\, f(x) \leq r,\,\, x\in U,\,\, r\in \mathbb{R}\}.
	\end{align*}
\end{definition}


\begin{remark}
	If $f:X \rightarrow Y$ is a function and $y\in  Y$, we denote the inverse image $f^{-1}(y)$ by $\left\{ f=y \right\}$.
\end{remark}


\begin{theorem}
	\label{teo_convexo e epigrafo}
	Let $U\subset \mathbb{V}$ be a non-empty convex set and $f:U \rightarrow \mathbb{R}$.
	Then $f$ is a convex function iff $\epi f$ is a convex subset of $\mathbb{V} \times \mathbb{R}$. 
	In particular, if $r\in \mathbb{R}$ and $f$ is a convex function, then 
	\begin{align}
		\nonumber
		\epi f \cap (\mathbb{V} \times \{r\}) & =\{(x,\tilde{r}) \in \mathbb{V} \times \mathbb{R}:f(x)\leq \tilde{r}, x\in U, \,\, \tilde{r} \in \mathbb{R} \} \cap (\mathbb{V} \times \{r\}) \\
		& = \{x \in \mathbb{V}:f(x) \leq r, x \in U\} \times \{r\} \nonumber \\
		& \cong \{x \in \mathbb{V}:f(x) \leq r, x \in U\}\nonumber 
	\end{align}
	are convex sets.
\end{theorem}
\begin{proof}
	See Section 4 of \cite{Rockafellar}.
\end{proof}


\begin{definition} \label{Definicao de norma assimetrica} An {\rm asymmetric norm} on $\mathbb{V}$ is a function $F:\mathbb{V} \rightarrow [0,\infty)$ that satisfies the following conditions:
	\begin{enumerate}
		\item $F(y) = 0$ only if $y = 0;$
		\item $F(\mu y) = \mu F(y)$ for every $\mu \in [0, \infty)$ and $y \in \mathbb{V}$ (positive homogeneity);
		\item $F(y + z) \leq F(y) + F(z)$ for every $y, z \in \mathbb{V}$ (triangle inequality).
	\end{enumerate}
\end{definition}


\begin{definition}
	Let $F:\mathbb V \rightarrow [0,\infty)$ be an asymmetric norm.
	\begin{enumerate}
		\item The {\rm open ball} centered at $y$ and radius $r$ is the subset
		\begin{align*}
			B_F(y,r) = \{z \in \mathbb V : F(z-y) < r\}.
		\end{align*}
		\item The {\rm closed ball} centered at $y$ and radius $r$ is the subset
		\begin{align*}
			B_F[y,r] = \{z \in \mathbb V : F(z-y) \leq r\}.
		\end{align*}
		\item The {\rm sphere} centered at $y$ and radius $r$ is the subset
		\begin{align*}
			S_F[y,r] = \{z \in \mathbb V : F(z-y) = r\}.
		\end{align*}
	\end{enumerate}
\end{definition}


\begin{definition} \label{Fortemente convexa - definicao artigo R-H} Let $\check{F}$ be an asymmetric norm on $\mathbb{V}.$ We say that an asymmetric norm $F$ is {\rm strongly convex with respect to} $\check{F}$ if
	\begin{align*}
		F^2(z) \geq F^2(y) + \alpha(z-y) + \check{F}^2(z-y)
	\end{align*}
	for every $y,z \in \mathbb{V}$ and $\alpha \in \partial F^2(y),$ where $\partial F^2(y)$ is the {\rm subdifferential of $F^2$ at $y.$}
\end{definition}


\begin{remark} The Definition \ref{Fortemente convexa - definicao artigo R-H} does not depend on the asymmetric norm $\check{F}$ because all asymmetric norms on finite dimensional vector spaces are equivalent. For more details, see \cite{Buro} and Remark 2.12 in \cite{Fukuoka-Setti2}.
\end{remark}


\begin{theorem} \label{Fortemente convexa - Equivalencia} Let $F$ be an asymmetric norm on $\mathbb{V}$ and $\gamma$ be a positive constant. Then
	$F^2$ is strongly convex with modulus $\gamma$ (according to Definition \ref{Fortemente convexa - definição clássica}) if, and only if,
	$F$ is strongly convex with respect to $\sqrt{\frac{\gamma}{2}}\Vert \cdot \Vert$ (according to Definition \ref{Fortemente convexa - definicao artigo R-H}).
\end{theorem}
\begin{proof}
	See Theorem 3.2 of \cite{Fukuoka-Setti2}.
\end{proof}


\begin{definition}
	\label{norma de Minkowski}
	A function $F:\mathbb V \rightarrow [0, \infty)$ is a {\rm Minkowski norm} if 
	\begin{enumerate}
		\item $F$ is smooth in $\mathbb V\backslash\{0\}$;
		\item $F(\mu y) = \mu F(y)$ for every $\mu \in [0,\infty)$ and $y\in \mathbb V$;
		\item If $(y^1, \ldots, y^n)$ is a coordinate system of $\mathbb V$ with respect to a basis of $\mathbb V$, then the  $n\times n$ hessian matrix
		\begin{align*}
			\left( g_{ij(y)} \right) := \left( \left[\frac{1}{2} F^2 (y) \right]_{y^iy^j} \right), \hspace{5mm} i,j=1, \ldots, n,
		\end{align*}
		is positive definite for every $y \in \mathbb V\backslash \{0\}$, where the subscript $y^iy^j$ stands for the partial derivatives with respect to $y^i$ and $y^j$.
	\end{enumerate}
\end{definition}

\begin{proposition} Every Minkowski norm is an asymmetric norm.
	\begin{proof}
		See Theorem 1.2.2 of \cite{BaoChernShen}.
	\end{proof}
\end{proposition}


\subsection{Mollifier smoothing of strongly convex functions}
\label{subsection_mollifier smoothing of convex functions}

In this subsection, we recall the theory of mollifier smoothing of real valued continuous functions defined on open subsets of a Euclidean space.
Proposition \ref{funcao convergencia uniforme} states that a mollifier smoothing of $f$ converges uniformly to $f$ on compact subsets, even when the smoothing is made only with respect to some variables.
Lemma \ref{Suavização de função fortemente convexa} states that the mollifier smoothing of a strongly convex function with modulus $\gamma$ is also a strongly convex function with modulus $\gamma$. For more details, see \cite{Fukuoka-Setti}, \cite{Fukuoka-Setti2} and Appendix C of \cite{Evans}.


Let $U$ be an open subset of $\mathbb{R}^n,$ $\varepsilon>0$ and $U_{\varepsilon}=\{y\in U : \dist(y,\partial U)>\varepsilon\},$ where $\partial U$ is
the boundary of $U.$ Let $\Vert \cdot \Vert$ be the Euclidean norm on $\mathbb{R}^n.$


\begin{definition} \label{definition de mollifier}  
	
	\
	
	\begin{enumerate}
		
		\item Define $\eta \in \mathit{C}^{\infty}(\mathbb{R}^n)$ by
		\begin{align*}
			\eta(y) & :=  \begin{cases}
				C\exp\left(\frac{1}{ \|y\|^2-1}\right)  \quad & {\rm{if  \quad }}\|y\|<1  \\
				0   & {\rm{if}} \quad \|y\|\geq 1,
			\end{cases}
		\end{align*}
		where the constant $C> 0$ is chosen so that $\int_{\mathbb{R}^n}\eta \ dx=1.$
		
		\item  For  each $\varepsilon>0,$ define
		\begin{align}
			\label{epsilon-mollifier}
			\eta_{\varepsilon}(y) & :=  \frac{1}{\varepsilon^n}\eta\left(\frac{y}{\varepsilon}\right).
		\end{align}
		
		We call $\eta$ of {\rm standard mollifier.} For the sake of simplicity, we will also call $\eta_\varepsilon$ the standard mollifier.
	\end{enumerate}
\end{definition} 


Note that $\eta_{\varepsilon}$ is $\mathit{C}^{\infty},$ 
\begin{align*}
	\int_{\mathbb{R}^n}\eta_{\varepsilon} \, dy = \int_{\mathbb{R}^n}\eta \, dy=1  \quad {\rm and} \quad {\rm{supp}}(\eta_{\varepsilon})=B_{\Vert \cdot \Vert}[0,\varepsilon],
\end{align*}
where $\rm{supp}(\eta_{\varepsilon})$ is the support of  $\eta_{\varepsilon}.$


\begin{definition} If $f:U \rightarrow \mathbb{R}$ is a continuous function, its {\rm mollifier smoothing} $\eta_\varepsilon \ast f$ is the convolution of $\eta_\varepsilon$ and $f$ in $U_\varepsilon.$ That is,
	\begin{align*}
		(\eta_{\varepsilon}*f)(y)  := \int_U\eta_{\varepsilon}(y-z)f(z) \, dz = \int_{B_{\Vert \cdot \Vert}[0,\varepsilon]}\eta_\varepsilon(z)f(y-z) \, dz,
	\end{align*}
	for every $y \in U_\varepsilon.$
\end{definition}


\begin{theorem} \label{propriedades da aplicacao mollifier} (Properties of mollifiers). Let $f:U\subset \mathbb{R}^n \rightarrow \mathbb{R}$ be a  continuous function, then the following statements are true:
	\begin{enumerate}
		\item \label{item_suavizacao suave}$(\eta_{\varepsilon}*f)\in \mathit{C}^{\infty}(U_\varepsilon)$.
		\item $(\eta_{\varepsilon}*f) \rightarrow f$ as $\varepsilon \rightarrow 0$ (pointwise).
		\item $(\eta_{\varepsilon}*f) \rightarrow f$ uniformly on compact subsets of $U.$
	\end{enumerate}
\end{theorem}


\begin{remark} \label{derivada mollifier} 
	A consequence of the proof of Item \eqref{item_suavizacao suave} of Theorem \ref{propriedades da aplicacao mollifier} (see  \cite{Evans}) is:
	\begin{align*}
		D^\alpha(\eta_{\varepsilon}*f)(y) =  \int D^\alpha\eta_{\varepsilon}(y-z)f(z) \, dz = (D^\alpha\eta_{\varepsilon}*f)(y),
	\end{align*}
	for every $y \in U_\varepsilon$ and for every multi-index $\alpha=(\alpha_1,\ldots,\alpha_n)$.
	Moreover, when $f$ is $\mathit{C}^\infty,$ it follows that
	\begin{align*}
		D^\alpha(\eta_{\varepsilon}*f)(y) =  \int\eta_{\varepsilon}(z)D^\alpha f(y-z) \, dz =  (\eta_{\varepsilon}*D^\alpha f)(y).
	\end{align*}
\end{remark}


\begin{definition} \label{Definição ast_1 e ast_2}
	Let $k\in \{1, 2, \ldots, n-1\}$ and $\varepsilon >0$. 
	Decompose $\mathbb R^n$ as $\mathbb R^k \times \mathbb R^{n-k}$ and denote its variables by $x=(x^1, \ldots, x^k)$ and $y = (x^{k+1}, \ldots, x^n)$. 
	Let $U$ be an open subset of $\mathbb R^k$ and $V$ be an open subset of $\mathbb R^{n-k}$. 
	Consider a continuous function $f: U \times V \rightarrow \mathbb R$ and let $\eta_\varepsilon : B_{\Vert \cdot \Vert}(0,\varepsilon) \subset \mathbb R^k \rightarrow \mathbb R$ as in (\ref{epsilon-mollifier}). 
	The mollifier smoothing $(\eta_\varepsilon \ast_1 f): U_\varepsilon \times V \rightarrow \mathbb R$ of $f$ with respect to the first $k$ variables is defined as
	\begin{align*}
		(\eta_\varepsilon \ast_1 f) (x,y) = \int \eta_\varepsilon (z) f(x - z, y) \, dz.
	\end{align*}
	The mollifier smoothing $(\eta_\varepsilon \ast_2 f)$ with respect to the last $n-k$ variables is defined analogously.
\end{definition}


\begin{proposition}
	\label{funcao convergencia uniforme}
	Let $U$ be an open subset of $\mathbb R^k$, $V$ be an open subset of $\mathbb R^{n-k}$ and $f: U \times V \rightarrow \mathbb R$ be a continuous function. Then $(\eta_\varepsilon \ast_1 f)$ converges uniformly to $f$ on compact subsets of $U \times V$. 
	The same result holds for $(\eta_\varepsilon \ast_2 f)$.
\end{proposition}

\begin{proof} See Propostion 1 of \cite{Fukuoka-Setti}.
\end{proof}


\begin{lemma} \label{Suavização de função fortemente convexa} If $f:\mathbb{R}^n \rightarrow \mathbb{R}$ is a strongly convex function with modulus $\gamma$, then $\eta_\varepsilon \ast f$ is also strongly convex with modulus $\gamma.$
\end{lemma}
\begin{proof}
	See Lemma 2.39 of \cite{Fukuoka-Setti2}.
\end{proof}


\subsection{Finsler geometry}
\label{subsection Finsler geometry}

In this subsection, we present some objects of Finsler geometry. In particular, we introduce the fundamental tensor and the geometrical objects necessary to define the Chern, Cartan, Hashiguchi and Berwald connections, as well as the flag curvature. For more details, see \cite{BaoChernShen}.

Before Finsler geometry, we recall how the operation of raising and lowering indices in tensors will be denoted. For more details, see \cite{Dubrovin-Fomenko-Novikov}.
Let $\mathbb V$ be a real vector space and $\mathcal B=\{e_1, \ldots,$ $e_n\}$ a basis of $\mathbb V$. 
Denote by $\mathbb V^\ast$ the dual vector space of $\mathbb V$ and by $\mathcal B^\ast = \{e^1, \ldots, e^n\}$ the dual basis of $\mathcal B$.
A $p$ times contravariant and $q$ times covariant tensor $T$ on $\mathbb V$ is represented by
\begin{align*}
	T^{i_1 \ldots i_p}_{j_1 \ldots j_q}e_{i_1}\otimes \ldots \otimes e_{i_p} \otimes e^{j_1} \otimes \ldots \otimes e^{j_q}
\end{align*}
or simply by $T^{i_1 \ldots i_p}_{j_1 \ldots j_q}$.
When $\mathbb V$ is endowed with a inner product $g=(g_{ij})$, then $(g^{ij})$ represent the inverse tensor of $g$ and the operation of raising and lowering indices are denoted by $\rho^j=g^{ji}\rho_i$ and $\xi_j = g_{ij}\xi^i$ respectively.


\begin{definition} \label{definição de estrutura de Finlser} A {\rm Finsler structure of} $M$ is a function $F:TM \to [0, \infty)$ with the following properties:
	\begin{enumerate}
		\item $F$ is $C^\infty$ on the entire slit tangent bundle $TM\setminus 0 := \{(x, y) \in TM : x \in M, y \, \neq 0\};$
		\item For every $x \in M,$ $F(x, \cdot):T_xM \to [0, \infty)$ is a Minkowski norm.
	\end{enumerate}
	The pair $(M, F)$ is called a {\rm Finsler manifold}.
\end{definition}

Let $(M,F)$ be a Finsler manifold and consider the natural projection $\pi: TM\setminus 0 \rightarrow M$. The pulled-back tangent bundle  
\begin{align*}
	\pi^* TM = \bigcup_{x\in M} \left( \{x\} \times (T_xM\setminus\{0\})\times T_xM  \right)
\end{align*}
and the pulled-back cotangent bundle  
\begin{align*}
	\pi^* T^*M = \bigcup_{x\in M} \left( \{x\} \times (T_xM\setminus\{0\})\times T_x^*M  \right)
\end{align*}
are vector bundles over the slit tangent bundle $TM\setminus 0$. We consider the identifications $(\pi^\ast TM )_{(x,y)} \cong T_xM$ and $(\pi^\ast T^\ast M )_{(x,y)} \cong T^\ast_xM$ on the fibers of these pulled-back bundles.

Let  $\x = (x^i)$ be a coordinate system of an open subset $U \subset M$ and $(\x, \y) = ((x^i),(y^i))$ be the natural coordinate system on $TU$ induced by $(x^i)$. 
The sections $\left\{ \frac{\partial}{\partial x^i} \right\}$ and $ \{ dx^i\}$ on $\pi^*TU$ and $\pi^*T^*U$ are respectively given by
\begin{align*}
	\frac{\partial}{\partial x^i} : TU\setminus 0 \rightarrow \pi^*TU, & & \frac{\partial}{\partial x^i}(x,y) = \left( x, y, \frac{\partial}{\partial x^i}\Big\vert_x \right),\\
	d x^i : TU\setminus 0 \rightarrow \pi^*T^*U, & & d x^i(x,y) = \left( x, y, d(x^i)_x \right).
\end{align*}
These sections are defined locally in $x$ and globally in $y$, and are called {\it transplanted sections}.

\begin{definition} 
	\label{define secao distinguida} 
	The {\rm distinguished section} $\ell$ is a section of $\pi^*TM$ defined by
	\begin{align}
		\ell \ = \ \ell_{(x,y)} \ := \ \frac{y^i}{F(x,y)} \, \frac{\partial}{\partial x^i}\Big\vert_x \ = \ \frac{y^i}{F}  \,  \frac{\partial}{\partial x^i} \ =: \ \ell^i  \,  \frac{\partial}{\partial x^i}.
		\label{define distinguished} 
	\end{align} 
\end{definition}

\begin{definition} 
	The {\rm fundamental tensor} is a section of inner products on $\otimes^2 \pi^\ast T^\ast  M$ given locally by
	\begin{align} \label{define tensor fundamental} 
		g_{ij(x,y)} = \frac{1}{2} \, \left( \frac{\partial^2 F^2}{\partial y^i \partial y^j}\right)(x,y).
	\end{align}
\end{definition}

\begin{proposition} \label{g ij redimensionamento} Let $F:TM \rightarrow [0,\infty)$ be a Finsler structure. Given $(x,y) \in TM\setminus 0$, we have that
	$g_{ij(x,y)}=g_{ij(x,\mu y)}$ for every $\mu > 0$. 
\end{proposition}

\begin{proof}
	See Chapter 1 of \cite{BaoChernShen}.
\end{proof}

\begin{definition} 
	\label{define tensor de cartan} 
	Let $(M,F)$ be a Finsler manifold. 
	The {\rm Cartan tensor} is symmetric section $A$ on $\otimes^3 \pi^\ast T^\ast M$ defined locally by 
	\begin{align}
		\label{define componente tensor de cartan}
		A_{(x,y)} = A_{ijk(x,y)} dx^i \otimes dx^j \otimes dx^k
		:= \frac{F}{2} \, \frac{\partial g_{ij}}{\partial y^k} \, dx^i \otimes dx^j \otimes dx^k.
	\end{align}
\end{definition}

The tensor 
\begin{align} 
	\label{unormalized Cartan tensor}
	C_{ijk}=\frac{1}{2} \, \frac{\partial g_{ij}}{\partial y^k} = \frac{1}{F} \, A_{ijk}
\end{align}
is also called the {\it Cartan tensor} in the literature.

The {\it formal Christoffel symbols} of the second kind are defined as 
\begin{align} 
	\gamma^i{}_{jk} := g^{is} \, \frac{1}{2} \, \left(\frac{\partial g_{sj}}{\partial x^k}-\frac{\partial g_{jk}}{\partial x^s}+\frac{\partial g_{ks}}{\partial x^j}  \right).
	\label{definicao formal Christoffel simbols}
\end{align}

The {\it nonlinear connection} is given by
\begin{align}
	N^i{}_j := \gamma^i{}_{jk} \, y^k- C^i{}_{jk} \, \gamma^k{}_{rs} \, y^r \, y^s,
	\label{definicao non linear connection}
\end{align}
where $C^i{}_{jk} = g^{is}C_{sjk}$. 

It is usual to consider the basis $\left\{ \frac{\delta}{\delta x^i}, F\frac{\partial}{\partial y^i}\right\}$ and $\left\{ dx^i, \frac{\delta y^i}{F}\right\}$ on $TU\setminus 0$ and $T^\ast U$ respectively, where
\begin{align*}
	\label{define delta delta}
	\frac{\delta}{\delta x^j} := \frac{\partial}{\partial x^j}-N^i{}_j \, \frac{\partial}{\partial y^i}
\end{align*}
and 
\begin{align*}
	\delta y^i:= dy^i+N^i{}_j \, dx^j.
\end{align*}

The Chern connection is characterized by the following theorem.

\begin{theorem} 
	\label{teorema conexao de Chern}
	Let $(M,F)$ be a Finsler manifold. The pulled-back bundle  $\pi^*TM$ admits a unique linear connection $\nabla$, called the {\rm  Chern connection}, such that its connection forms $\omega_j^{\ i}$, defined by
	\begin{align*}
		\nabla_X \, \frac{\partial}{\partial x^j} = \omega_j{}^i(X) \, \frac{\partial}{\partial x^i},
	\end{align*}
	are locally characterized by the following structural equations:
	\begin{itemize}
		\item {\bf Torsion freeness:}
		\begin{align*}
			d(dx^i)-dx^j\wedge \omega_j{}^i = - dx^j \wedge \omega_j{}^i=0.
		\end{align*}
		\item {\bf Almost $g$-compatibility:} 
		\begin{align*}
			dg_{ij}- g_{kj} \, \omega_i{}^k-g_{ik} \, \omega_j{}^k = 2 \, C_{ijs} \, \delta y^s.
		\end{align*}
	\end{itemize}
\end{theorem}

\begin{proof}
	See Chapter 2 of \cite{BaoChernShen}.
\end{proof}

The torsion freeness is equivalent to the absence of $dy^k$ in  $\omega_j{}^i,$ that is,
\[
\omega_j{}^{i}=\Gamma^i{}_{jk} \, dx^k,
\]
together with the symmetry
\begin{align*}
	\Gamma^i{}_{jk}=\Gamma^i{}_{kj}.
\end{align*}
The almost metric-compatibility  implies that
\begin{align} \label{Gamma almost metric compatibility}
	\Gamma^l{}_{jk}= \gamma^l{}_{jk}-g^{li} \, \left( C_{ijs} \, N^s{}_k- C_{jks} \, N^s{}_i+ C_{kis} \, N^s{}_j \right),
\end{align}
or equivalently
\begin{align*}
	\Gamma^i{}_{jk}= \frac{g^{is}}{2} \, \left(\frac{\delta g_{sj}}{\delta x^k}-\frac{\delta g_{jk}}{\delta x^s}+\frac{\delta g_{ks}}{\delta x^j} \right).
\end{align*}

Denote by
\begin{align*}
	\dot A = \nabla_{\ell^i \frac{\delta}{\delta x^i}} \, A,
\end{align*}
the horizontal covariant derivative of the Cartan tensor $A$ along the distinguished (horizontal) direction $\hat{\ell} := \ell^i \, \frac{\delta}{\delta x^i}$.
Then
\begin{align}
	\label{define a dot}
	\dot A^i{}_{jk} = \frac{1}{2} \, (\sigma^i)_{y^jy^k} - \Gamma^i{}_{jk}
\end{align}
where
\begin{align}
	\label{definicao Gi}
	\sigma^i = \gamma^i{}_{jk} \, y^j \, y^k 
\end{align}
(see \cite{BaoChernShen}).

The curvature 2-forms of the Chern connection are defined by
\begin{align*}
	\Omega_j{}^i:=  d\omega_j{}^i - \omega_j{}^k\wedge \omega_k{}^i. 
\end{align*}
Writing the $2$-forms $\Omega_j^{\ i}$ in terms of the basis  $\left\{dx^k,\frac{\delta y^k}{F}\right\},$ we have that
\begin{align*}
	\Omega_j{}^i = \frac{1}{2} \, R_j{}^i{}_{kl} \, dx^k\wedge dx^l+P_j{}^i{}_{kl} \, dx^k \wedge \frac{\delta y^l}{F}+\frac{1}{2} \, Q_j{}^i{}_{kl} \, \frac{\delta y^k}{F} \wedge \frac{\delta y^l}{F}.
\end{align*}
The objects $R,P$ and $Q$ are respectively the $hh$-, $hv$-, $vv$-curvature tensors of the Chern connection and they can be written as
\begin{align*}
	R_j{}^i{}_{kl} & = \frac{\delta \Gamma^i{}_{jl}}{\delta x^k}-\frac{\delta \Gamma^i{}_{jk}}{\delta x^l}+ \Gamma^i{}_{hk} \, \Gamma^h{}_{jl}-\Gamma^i{}_{hl} \, \Gamma^h{}_{jk}, \label{Rjikl hh} \\
	P_j{}^i{}_{kl} & = -  F \, \frac{\partial \Gamma^i{}_{jk}}{\partial y^l}, \nonumber \\
	Q_j{}^i{}_{lk} & = 0. \nonumber
\end{align*}
Notice that
\[
R_j{}^i{}_{kl}  = - R_j{}^i{}_{lk}.
\]

Now let us define the flag curvature. 
A {\it flag} at $x \in M$ consists of two linearly independent vectors in $T_xM$, a vector $y \in T_xM,$ which is called the {\it flagpole}, and $z:=z^i\frac{\partial}{\partial x^i},$ which is the {\it transversal edge} to $y$.

The flag curvature is the correspondence which associates to each $x \in M$ and each flag $(y,z) $ in $T_xM$  the number
\begin{align*}
	K(y,z):=\frac{z^i \, (y^j \, R_{jikl} \, y^l) \, z^k}{g(y,y) \, g(z,z)-[g(y,z)]^2},
\end{align*}
where 
\begin{align} \label{Rjikl abaixa}
	R_{jikl} = g_{im}R_j{}^m{}_{kl}.
\end{align}
$K(y,z)$ depends only on $y$ and $\spann\{y,z\}$.

\vspace{0.25cm}

Now let us introduce other connections:

\begin{definition}
	\label{Cartan Hasiguchi Berwarld} The connection forms of the following connections on $\pi^\ast TM$ are defined by:
	\begin{enumerate}
		\item {\rm Cartan connection}: $\omega_j{}^i + C^i{}_{jk} \, \delta y^k$.
		\item {\rm Hashiguchi connection}: $\omega_j{}^i + C^i{}_{jk} \, \delta y^k + \dot A^i{}_{jk} \, dx^k$.
		\item {\rm Berwald connection}: $\omega_j{}^i + \dot A^i{}_{jk} \, dx^k$.
	\end{enumerate}
\end{definition}

\begin{remark}
	The flag curvature for the Chern connection, Cartan connection, Hashiguchi connection and Berwald connection are the same.
\end{remark}

We conclude this subsection by presenting the definition of $C^0$-Finsler structure, which is a natural generalization of Finsler structure.

\begin{definition} 
	\label{define estutura de Finsler C0}
	A {\rm $C^0$-Finsler structure of} $M$ is a continuous function $F:TM \to[0, \infty)$ such that $F(x,\cdot): T_xM \to [0, \infty)$ is an asymmetric norm for every $x \in M.$
\end{definition}


\section{Mollifier smoothing of strongly convex $C^0$-Finsler structures on smooth manifolds}
\label{section Mollifier smoothing of strongly convex $C^0$-Finsler structures on smooth manifolds}

Subsections \ref{section vertical smoothing} and \ref{secao mollifier smoothing F} are dedicated to constructing the mollifier smoothing, $F_\varepsilon:TM \to [0, \infty),$ of a strongly convex $C^0$-Finsler structure $F:TM \to [0, \infty)$.
In Subsection \ref{secao convergencia conexoes curvatura}, we study the convergence, for example of the fundamental tensor, the Chern connection and the flag curvature of $(M,F_\varepsilon)$ to the corresponding objects of $(M,F),$ when $F$ is a Finsler structure.
In Subsection \ref{Example}, we prove that if $F:TM \to [0, \infty)$ is a strongly convex $C^0$-Finsler structure, there does not always exist a $\gamma>0$ such that $F(x, \cdot)$ is strongly convex with respect to $\sqrt{\frac{\gamma}{2}}\Vert \cdot \Vert$ for every $x$ in a neighborhood $U$ of $M$. 


\subsection{Vertical smoothing of strongly convex $C^0$-Finsler structures}
\label{section vertical smoothing}

In this subsection, we will define strongly convex $C^0$-Finsler structures and the vertical smoothing of such structures.

\begin{definition} 
	Let $\x : U \to \mathbb{R}^n$ be a local coordinate system on an open subset $U$ of $M$. 
	The {\rm Riemannian metric $\left< \cdot, \cdot \right>^{\x}$ on $U$ induced by $\x$} is the $\x-$pull-back of the canonical Euclidean metric of $\x(U)$ on $U$, that is,
	\begin{align*}
		\left< y , z \right>^{\x}_x = \left< d\x_x(y), d\x_x(z)\right>,
	\end{align*}
	for every $x \in U$ and $y,z \in T_xM$.
\end{definition}


\begin{remark}
	If there isn't any possibility of confusion, we denote the Riemannian metric induced by $\x$ restricted to $T_xM$ also by $\left< \cdot, \cdot \right>$, and the corresponding norm by $\Vert \cdot \Vert$.
\end{remark}


\begin{definition} \label{Definição de estrutura de Finsler de classe C^0 fortemente convexa} 
	A $C^0$-Finsler structure $F:TM \to [0, \infty)$ is {\rm strongly convex} if $F(x, \cdot):T_xM \to [0, \infty)$ is strongly convex for each $x \in M$ (according to Definition \ref{Fortemente convexa - definicao artigo R-H}).
\end{definition}


Let $F:TM \to [0, \infty)$ be a strongly convex $C^0$-Finsler structure. The construction of vertical smoothing $\tilde{F}_ \varepsilon$ of $F$ will be done in a coordinate system.


Let
\begin{align*} 
	\x = (x^i): U_0 \to \mathbb{R}^n
\end{align*} be a coordinate system on an open subset $U_0$ of $M$. Let $U$ be an open subset of $U_0$ such that
\begin{align} \label{definindo U}
	U \Subset U_0.
\end{align}


Let
\begin{align} 
	\label{sistema de coordenadas x,y}
	(\x,\y) = ((x^i), (y^i))
\end{align}
be the natural coordinates on $TU$ with respect to $\boldsymbol{x}$.


For each $x \in U$ there exists a $\gamma_{x} > 0$ such that $F(x, \cdot)$ is strongly convex with respect to $\sqrt{\frac{\gamma_{x}}{2}}\Vert \cdot \Vert$, that is,
\begin{align*}
	F^2(x, z) \geq F^2(x, y) + \alpha(x, z-y) + \frac{\gamma_{x}}{2}\Vert z-y \Vert^2
\end{align*}
for every $y,z \in T_xM$ and $\alpha(x, \cdot) \in \partial F^2(x,y),$ where $\partial F^2(x,y)$ is the {\it subdifferential of $F^2(x, \cdot)$ at $y.$}
By the equivalence given in Theorem \ref{Fortemente convexa - Equivalencia}, we have that $F^2(x, \cdot)$ is strongly convex with modulus $\gamma_{x}$ for each $x \in U,$ that is,
\begin{align*}
	F^2(x, ty+(1-t)z) \leq tF^2(x, y)+(1-t)F^2(x, z)-\frac{1}{2}\gamma_x t(1-t) \Vert y-z \Vert^2,
\end{align*}
for every $y,z \in T_xM$ and $t \in [0, 1].$


We will work with two asymmetric norms in each tangent space $T_xU$: 
$F(x, \cdot)$ and $\Vert \cdot \Vert_x^{\x}$. For instance the sphere centered at $v$ and radius $r$ in $T_xU$ with respect to norm $F(x, \cdot)$ is denoted by $S_F[x, v, r]$ and so on.


Define
\begin{align} 
	\label{constante r_U}
	r_U = \max_{(x,y) \in \, S_{\Vert \cdot \Vert}M|_{\overline{U}}} F(x,y) = \max_{(x,y) \in \, B_{\Vert \cdot \Vert}M|_{\overline{U}}} F(x,y)
\end{align}
and
\begin{align} 
	\label{constante rho_U}
	r_u= \min_{(x,y) \in \, S_{\Vert \cdot \Vert}M|_{\overline{U}}} F(x,y),
\end{align}
where
\begin{align*} 
	S_{\Vert \cdot \Vert}M|_{\overline{U}} & = \left\{ (x,y) \in TM : x \in \overline{U}, \Vert y \Vert_x = 1 \right\} \quad {\rm and} \\ 
	B_{\Vert \cdot \Vert}M|_{\overline{U}} & = \left\{ (x,y) \in TM : x \in \overline{U}, \Vert y \Vert_x \leq 1 \right\}.
\end{align*}


Let
\begin{align*} 
	\mathcal A := \left\{(x,y) \in TU : x \in U, \Vert y \Vert \in \left[ \frac{1}{2}, \frac{2r_U}{r_u} \right] \right\}
\end{align*}
which contains $F^{-1}([r_U/2,2r_U])\cap TU$.


The next lemma is a direct consequence of Lemma 5 of \cite{Fukuoka-Setti}. 


\begin{lemma} 
	\label{lema F em TU}
	Let $F:TM \to [0, \infty)$ be a strongly convex $C^0$-Finsler structure, $U$ be an open subset satisfying \eqref{definindo U} and  $(\x, \y) = ((x^i),(y^i))$ be a coordinate system on $TU$ as defined in \eqref{sistema de coordenadas x,y}. Let $v\in {\mathbb R^n}$ be a nonzero vector and consider
	\begin{align*} 
		\varepsilon \in \left(0,\frac{r_u \Vert v \Vert}{4nr_U}\right],
	\end{align*}
	where $r_U$ and $r_u$ are defined by \eqref{constante r_U} and \eqref{constante rho_U}, respectively. Then
	\begin{align} \label{Desigualdade para F^2}
		F^2\left(x,w+t\frac{v}{\Vert v \Vert}\right) - F^2(x,w) > \frac{r_u^2}{32}t 
	\end{align}
	for every $x \in U,$ $w \in B_{\Vert \cdot \Vert}[x, v,\varepsilon]$ and $t>0$.
\end{lemma}
\begin{proof}
	The proof of Inequality \eqref{Desigualdade para F^2} follows from the difference-of-squares identity and the inequality
	\begin{align*} 
		F\left(x,w+t\frac{v}{\Vert v \Vert}\right) - F(x,w) > \frac{r_u}{8}t, 
	\end{align*}
	for every $x \in U,$ $w \in B_{\Vert \cdot \Vert}[x, v,\varepsilon]$ and $t>0$, contained in Lemma 5 of \cite{Fukuoka-Setti} (cf. Lemma 2.22 of \cite{Fukuoka-Setti2}).
\end{proof}


The {\it ``vertical" mollifier smoothing of} $F^2$ is defined by
\begin{align} \label{eta ast_v F^2}
	(\eta_\varepsilon \ast_v F^2)(x, y) = \int \eta_\varepsilon(z) F^2(x, y-z) \, dz, 
\end{align}
for $(x, y) \in TU \cong \x(U) \times \mathbb{R}^n$ and $\varepsilon \in (0, \tau),$ where
\begin{align} \label{define-tau}
	\tau \in \left(0, \frac{r_u}{8nr_U}\right),
\end{align} 
$r_U$ and $r_u$ are defined as in \eqref{constante r_U} and \eqref{constante rho_U}, respectively, and $\ast_v \coloneqq \ast_2$ represents the mollifier smoothing of the vertical part along the fibers of $TM$ (see Definition \ref{Definição ast_1 e ast_2}).

It follows directly from \eqref{eta ast_v F^2} that $(\eta_\varepsilon \ast_v F^2)$ is continuous and from Proposition \ref{propriedades da aplicacao mollifier} that $(\eta_\varepsilon \ast_v F^2)(x,\cdot)$ is smooth for every $x\in U.$


A natural question is whether the square root of the function in \eqref{eta ast_v F^2} is the desired vertical smoothing of $F$, but it fails because $\sqrt{(\eta_\varepsilon \ast_v F^2)(x, \cdot)}$ is not an asymmetric norm for every $x \in U.$ In fact, $(\eta_\varepsilon \ast_v F^2)(x, 0) > 0$ for every $x \in U.$


The next lemma will be used to construct the spheres in each tangent space of the vertical mollifier smoothing. 
We use the notation $\Hess_y[(\eta_\varepsilon \ast_v F^2)(x,y)]$ in order to denote the Hessian of $(\eta_\varepsilon \ast_v F^2)$ with respect to $(y^1,\ldots,y^n)$.


\begin{lemma}
	\label{lema propriedade de F_epsilon} 
	$(\eta_\varepsilon \ast_v F^2)$ has the following properties for every $\varepsilon \in \left(0, \tau\right):$
	\begin{enumerate}
		\item $\Hess_y[(\eta_\varepsilon \ast_v F^2)(x,y)](v, v) \geq \gamma_x \Vert v \Vert^2$ for every $(x, y) \in TU$ and $v \in \mathbb{R}^n.$
		\item The {\rm radial derivative} of $\left(\eta_\varepsilon \ast_v F^2\right)$ on $\mathcal A$ satisfies
		\begin{align*}
			\frac{y^i}{\Vert y\Vert}\left(\eta_\varepsilon \ast_v F^2\right)_{y^i}(x, y) > \frac{r_u^2}{32},\,\, \forall \, (x, y) \in \mathcal A.
		\end{align*}
		\item $\left(\eta_\varepsilon \ast_v F^2\right)(x, y)<r_U^2$ for every $(x, y) \in TU$ with $\Vert y \Vert \leq \frac{1}{2}.$
		\item $\left(\eta_\varepsilon \ast_v F^2\right)(x, y)>r_U^2$ for every $(x, y) \in TU$ with $\Vert y \Vert \geq \frac{2r_U}{r_u}.$
	\end{enumerate}
	In particular, we have that $\left\{ \eta_\varepsilon \ast_v F^2 = r_U^2\right\}\cap T_xU$ is a smooth hypersurface in $T_xU$ such that its radial projection onto  $S_{\|\cdot\|}[x,0,1]$ is a diffeomorphism.
\end{lemma}
\begin{proof}
	\
	\begin{enumerate}
		\item It follows directly from Theorem \ref{Equivalencias - Fortemente convexa (Definicao Classica)} and Lemma \ref{Suavização de função fortemente convexa}.
		\item Consider $(x, y) \in \mathcal{A}.$ Then
		\begin{align*}
			& \left(\eta_\varepsilon \ast_v F^2\right) \left(x, y+t\frac{y}{\|y\|}\right)-\left(\eta_\varepsilon \ast_v F^2\right) (x, y) \\
			= & \int_{B_{\|\cdot\|}[x,0,\varepsilon]} \eta_\varepsilon(z)\left[ F^2\left(x, y+t\frac{y}{\|y\|}-z\right) - F^2(x, y-z) \right]  dz \\
			= & \int_{B_{\|\cdot\|}[x,y,\varepsilon]} \eta_\varepsilon(y-z)\left[ F^2\left(x, z+t\frac{y}{\|y\|}\right) - F^2(x, z) \right] dz \\
			> &  \frac{r_u^2}{32}t,
		\end{align*}
		for $t\in (0,1)$, where the inequality is due to Lemma \ref{lema F em TU}. Thus, $\frac{r_u^2}{32}$  is a lower bound for $\frac{y^i}{\Vert y \Vert}(\eta_\varepsilon \ast_v F^2)_{y^i}(y)$.  
		\item Consider $(x, y) \in TU$ with $\Vert y \Vert \leq \frac{1}{2}.$ We have
		\begin{align*}
			\left(\eta_\varepsilon \ast_v F^2\right)(x, y) = \int \eta_\varepsilon(z)F^2\left(x, y-z\right) \, dz,
		\end{align*}
		with $z \in {B_{\|\cdot\|}[x,0,\varepsilon]}.$ Then $\Vert y-z \Vert < 1,$ what implies that $F^2(x, y-z) < r_U^2$, and therefore  $\left(\eta_\varepsilon \ast_v F^2\right)(x, y) < r_U^2.$
		\item Consider $(x, y) \in TU$ with $\Vert y \Vert \geq \frac{2r_U}{r_u}.$ We have that $\Vert y-z \Vert > 1$ for $z \in {B_{\|\cdot\|}[x,0,\varepsilon]},$ what implies $F^2\left(x, y-z\right) >  r_U^2.$ Therefore  $\left(\eta_\varepsilon \ast_v F^2\right)(x, y) > r_U^2.$
	\end{enumerate}
\end{proof}

Let $\varepsilon \in (0, \tau).$ For every $x \in U$, Theorem \ref{teo_convexo e epigrafo} and Lemmas \ref{Suavização de função fortemente convexa} and \ref{lema propriedade de F_epsilon} implies that 
\begin{align}
	\label{equation_bola aberta Ftilepsilon}
	\left\{ y \in T_xU:  (\eta_\varepsilon \ast_v F^2)(x, y) \leq r_U^2\right\} 
\end{align}
is bounded, closed and convex. 
Item (3) of Lemma \ref{lema propriedade de F_epsilon} implies that \eqref{equation_bola aberta Ftilepsilon} contains the origin in its interior. 


Consider the application
\begin{align} \label{aplicacao F_epsilon til}
	\tilde{F}_\varepsilon: TU \rightarrow [0, \infty)
\end{align}
such that restricted to each tangent space $T_xU$ is an asymmetric norm with the sphere of radius $r_U$ equal to the boundary of the set \eqref{equation_bola aberta Ftilepsilon}. 

It is appropriate here to compare $\tilde{F}_{\varepsilon}$ with the vertical smoothing done in \cite{Fukuoka-Setti}.
For $x \in U$,
what determines $\tilde{F}_{\varepsilon}$ is the family of spheres $S_{\tilde{F}_{\varepsilon}}[x,0,r_U] = \{y \in T_xU: (\eta_{\varepsilon} \ast_{v} F^2)(x,y) = r_{U}^{2}\}$, whereas the vertical smoothing $G_{\varepsilon}$ in \cite{Fukuoka-Setti} is determined by the family of spheres 
$S_{G_{\varepsilon}}[x,0,r_U] = \{y \in T_xU: (\zeta_{\varepsilon} \ast_{v} F)(x,y) = r_{U}\}$. 
Although the construction in \cite{Fukuoka-Setti} uses a perturbation $\zeta_{\varepsilon}$ of $\eta_{\varepsilon}$ and smooths $F$ instead of $F^2$, the resulting spheres behave analytically in the same way.

We follow \cite{Fukuoka-Setti} and obtain an expression for $\tilde{F}_{\varepsilon}$, analogous to $G_{\varepsilon}$ (see \eqref{F_epsilon til em coordenadas esfericas}).
Once this expression is obtained, the results follow analogously to those in \cite{Fukuoka-Setti}, and most of their proofs can be omitted.

First of all, we represent every sphere locally as the graph of a function.
Consider the identification $TU \cong (x^i)(U) \times \mathbb{R}^n$.
Let $S_{\theta}$ be an open subset of $S_{\Vert \cdot\Vert}[0,1]$, and $C(S_{\theta})$ the open cone of $S_{\theta}$ with vertex at the origin, where we can define spherical coordinates
\begin{align*}
	(r, \theta):=(r,\theta^2,\ldots,\theta^n): C(S_\theta) \to (0, \infty) \times W_\theta.
\end{align*}
Fix $x_0 \in U$. By Item (2) of the Lemma \ref{lema propriedade de F_epsilon} the radial derivative of
$(\eta_\varepsilon \ast_v F^2)(x_0, r_0,\theta_0)$ is positive for every $(r_0,\theta_0) \in C(S_\theta)$. 
Then, by the Implicit Function Theorem, the sphere $S_{\tilde{F}_\varepsilon}[x_0,0,r_U$] can be written locally as the graph of a smooth function
\begin{align*}
	r = \phi_\varepsilon(x_0,\theta) \in \left(\frac{1}{2}, \frac{2r_U}{r_u} \right),
\end{align*}
for every $\theta \in W_\theta$. 
Varying $x \in U$ and using the fact that $\tilde{F}_\varepsilon$ is positively homogeneous of degree 1, we have
\begin{align} 
	\label{F_epsilon til em coordenadas esfericas} 
	\tilde{F}_\varepsilon (x, r,\theta) = \frac{r \cdot r_U}{\phi_\varepsilon(x,\theta)}
\end{align}
for every  $(x, r,\theta) \in U \times (0,\infty) \times W_\theta.$


\begin{proposition}
	\label{propriedades da aplicacao F_epsilon^til}  
	The function $\tilde{F}_\varepsilon$ defined in \eqref{aplicacao F_epsilon til} satisfies the following properties for every $\varepsilon \in \left(0,\tau\right):$
	\begin{enumerate}
		\item $\tilde{F}_\varepsilon$ is a $C^0$-Finsler structure.
		\item $\tilde{F}_\varepsilon(x,\cdot)$ is smooth in $T_xU \setminus \{0\}.$
	\end{enumerate}
\end{proposition}

\begin{proof}
	The proof is similar to that of Proposition 3 in \cite{Fukuoka-Setti}.
\end{proof}


\begin{proposition}
	\label{g epsilon radialmente convexo}
	The function $\tilde{F}_\varepsilon(x,\cdot): T_xU \rightarrow [0,\infty)$ is a Minkowski norm for every $x \in U$ and $\varepsilon \in (0, \tau)$.
\end{proposition}

\begin{proof} 
	The proof is similar to that of Proposition 4 in \cite{Fukuoka-Setti}.
\end{proof}


\begin{theorem}
	\label{F_epsilon^til^2 converge uniforme}
	$\tilde{F}_\varepsilon{}^2 \rightarrow F^2 $ uniformly on compact subsets of $TU\cong U \times \mathbb R^n$.
\end{theorem}

\begin{proof} The proof is similar to that of Proposition 5 in \cite{Fukuoka-Setti}.
\end{proof}


\begin{corollary}
	\label{F_epsilon^til converge uniforme}
	$\tilde{F}_\varepsilon \rightarrow F $ uniformly on compact subsets of $TU\cong U \times \mathbb R^n$.
\end{corollary}

\begin{proof}
	It follows from Theorem \ref{F_epsilon^til^2 converge uniforme}, and the fact that the square root function is uniformly continuous on compact subsets of $[0, \infty)$.
\end{proof}


\subsection{Mollifier smoothing of strongly convex $C^0$-Finsler structures}
\label{secao mollifier smoothing F}

In this subsection, we define a mollifier smoothing $F_\varepsilon:TM\rightarrow [0, \infty)$ of a strongly convex $C^0$-Finsler structure $F$.
This smoothing will be constructed similarly to the mollifier smoothing constructed in Section 6 of \cite{Fukuoka-Setti}, which was done using smooth partition of unity.


Let 
\begin{align*}
	\{(U_\lambda,\x_\lambda)\}_{\lambda\in\Lambda}
\end{align*} 
be a locally finite smooth structure on $M$  such that $U_\lambda$ is defined as $U$ in \eqref{definindo U}. 
Let 
\begin{align*} 
	\{\varphi_\lambda\}_{\lambda\in \Lambda}
\end{align*}
be a smooth partition of unity on $M$ subordinate to the covering  $\{U_\lambda\}_{\lambda\in\Lambda}$ of $M.$ For each $\lambda \in \Lambda,$ denote $V_\lambda=\mathrm{int}(\supp(\varphi_\lambda))$ and $\nu_\lambda=\dist_{{\mathbb R^n}}(\partial V_\lambda,\partial U_\lambda).$


Fix $\lambda \in \Lambda$.  Let $r_{U_\lambda}$, $r_{u_\lambda}$ and $\tau_\lambda$ as defined in \eqref{constante r_U}, \eqref{constante rho_U} and \eqref{define-tau}, respectively. Let $\delta_\lambda = \min\{\tau_\lambda,\nu_\lambda\}.$ For every $\varepsilon \in (0,\delta_\lambda)$  consider the function
\begin{align*}
	\tilde{F}_{\varepsilon,\lambda}:TU_\lambda \to [0, \infty) 
\end{align*}
defined as in \eqref{aplicacao F_epsilon til}.
By Proposition \ref{propriedades da aplicacao F_epsilon^til}, the function $\tilde{F}_{\varepsilon,\lambda}$ is continuous, then we can consider the {\it``horizontal'' mollifier smoothing of} $\tilde{F}_{\varepsilon,\lambda}{}^2$ which is defined by
\begin{align*}
	\left(\eta_\varepsilon \ast_h \tilde{F}_{\varepsilon,\lambda}{}^2\right)(x,y) = \int \eta_{\varepsilon}(x-z) \tilde{F}_{\varepsilon,\lambda}{}^2(z,y) dz, \hspace{5mm}\forall (x,y) \in TO_{\varepsilon,\lambda},
\end{align*}
where $O_{\varepsilon, \lambda}$ is an open subset of $U_\lambda$ containing $\overline{V_\lambda}$ such that $\left(\eta_\varepsilon \ast_h \tilde{F}_{\varepsilon,\lambda}{}^2\right)$ is well defined, and $\ast_h \coloneqq \ast_1$ represents the mollifier smoothing of the horizontal part on an open subset of $M$ (see the Definition \ref{Definição ast_1 e ast_2}).

In order to make the variation of $\varepsilon$ independent of $\lambda,$ define $\psi_\lambda:(0,1) \rightarrow (0,\delta_\lambda)$ by $\psi_\lambda(\varepsilon)=\varepsilon\delta_\lambda$ and set
\begin{align*} 
	F_{\varepsilon,\lambda}{}^2(x,y)=  \left(\eta_{\psi_\lambda(\varepsilon)} \ast_h \tilde{F}_{\psi_\lambda(\varepsilon),\lambda}{}^2\right)(x,y), 
\end{align*} 
where $\varepsilon \in (0,1)$ and $(x,y) \in TO_{\varepsilon, \lambda}$.
Thus, for each $\varepsilon\in (0,1),$ we can define the function $F_\varepsilon:TM\rightarrow [0,\infty)$ implicitly by
\begin{align} 
	\label{F epsilon definida em TM}
	F_\varepsilon{}^2(x,y) = \sum_{\lambda \in \Lambda} \varphi_\lambda(x) F_{\varepsilon,\lambda}{}^2(x,y).
\end{align}


\begin{theorem}
	\label{propriedades da aplicacao F epsilon definida em TM} 
	The function $F_\varepsilon: TM \rightarrow \mathbb [0, \infty)$ defined by \eqref{F epsilon definida em TM} satisfies the following properties for every $\varepsilon \in (0, 1)$:
	\begin{enumerate}
		\item $ {F}_{\varepsilon }$ is a Finsler structure of $M$ for every $\varepsilon \in (0,1).$
		\item $ {F}_{\varepsilon } \rightarrow F$ uniformly on compact subsets of $TM$ as $\varepsilon$ converges to zero.
	\end{enumerate}
\end{theorem}

\begin{proof} The proof is similar to that of Theorem 4 in \cite{Fukuoka-Setti}.
\end{proof}


\begin{definition}
	The one parameter family of applications $F_\varepsilon$ given by (\ref{F epsilon definida em TM}) is a {\rm mollifier smoothing} of a strongly convex $C^0$-Finsler structure $F$.
\end{definition}


\subsection{Convergence of connections and curvatures}
\label{secao convergencia conexoes curvatura}

In this subsection, $(M, F)$ is a Finsler manifold. Let $F_\varepsilon: TM \to [0, \infty)$ be a mollifier smoothing of $F$ defined as in \eqref{F epsilon definida em TM}. We will study the convergence of the fundamental tensor, Chern connection, Cartan connection, Hashiguchi connection, Berwald connection and flag curvature of $(M, F_\varepsilon)$ to the corresponding objects of $(M, F)$. We represent geometric objects with respect to $F_\varepsilon$ with a subscript $\varepsilon$: for instance the fundamental tensor of $F_\varepsilon$ is denoted by $g_\varepsilon$.


The next lemma proves that the partial derivatives of objects with respect to $F_\varepsilon$ presented in the Equations \eqref{define distinguished}--\eqref{Rjikl abaixa} converge uniformly to the respective objects of $F$.


\begin{lemma}
	\label{several uniform convergences}
	Let $(M,F)$ be a Finsler manifold and $F_\varepsilon:TM \rightarrow [0,\infty)$ be a mollifier smoothing of $F$ as defined in \eqref{F epsilon definida em TM}.
	Let $U$ be an open subset of $M$ as defined in (\ref{definindo U}) and $(\x, \y) = ((x^i), (y^i))$ the coordinate system on $TU$ as defined in (\ref{sistema de coordenadas x,y}). For every multi-index $\alpha$, the following geometrical objects of $(M,F_\varepsilon)$ converges uniformly to the respective objects of $(M,F)$ on compact subsets of $TU\backslash 0$:
	\begin{align}
		D^\alpha (g_\varepsilon)_{ij} & \rightarrow D^\alpha g_{ij}. \nonumber \\
		D^\alpha (g_\varepsilon)^{ij} & \rightarrow D^\alpha g^{ij}. \nonumber \\
		D^\alpha (\ell_\varepsilon)^i & \rightarrow D^\alpha \ell^i. \nonumber \\
		D^\alpha (A_\varepsilon)_{ijk} & \rightarrow D^\alpha A_{ijk}. \nonumber \\
		D^\alpha (C_\varepsilon)_{ijk} & \rightarrow D^\alpha C_{ijk}. \nonumber \\
		D^\alpha (\gamma_\varepsilon)^i{}_{jk} & \rightarrow D^\alpha \gamma^i{}_{jk}. \nonumber \\
		D^\alpha (N_\varepsilon)^i{}_j & \rightarrow D^\alpha N^i{}_j. \nonumber \\
		D^\alpha (\Gamma_\varepsilon)^i{}_{jk} & \rightarrow D^\alpha \Gamma^i{}_{jk}. \nonumber \\
		D^\alpha (\sigma_\varepsilon)^i & \rightarrow D^\alpha \sigma^i. \nonumber \\
		D^\alpha (\dot A_\varepsilon)^i{}_{jk} & \rightarrow D^\alpha \dot A^i{}_{jk}. \nonumber \\
		D^\alpha (R_\varepsilon)_{jikl} & \rightarrow D^\alpha R_{jikl}. \nonumber
	\end{align}
\end{lemma}
\begin{proof} See Theorem 6 and Lemma 12 of \cite{Fukuoka-Setti}. 
\end{proof}


Now we study the uniform convergence of the flag curvature and the several connections of Finsler geometry, in such a way that the convergence doesn't depend on the choice of coordinate system.
We proceed similarly to \cite{Fukuoka-Setti}, and we recall some definitions from there for the sake of completeness.

\begin{definition}
	A {\rm$C^0$-Finsler structure on} $\pi^\ast TM$ is a continuous function $F_{\pi}:\pi^* TM \rightarrow [0, \infty)$ such that $F_{\pi}(x,y,\cdot):(\pi^\ast TM)_{(x,y)} \rightarrow \mathbb [0, \infty)$ is an asymmetric norm on $(\pi^\ast TM)_{(x,y)} \cong T_xM$.
\end{definition}

\begin{definition}
	\label{definition_uniform_convergence}
	Let $M$ be a smooth manifold and $K_{TM\backslash 0}$ be a compact subset of $TM\backslash 0$.
	We say that a one-parameter family of continuous sections $E_\varepsilon$ converges uniformly to a continuous section $E$ on $K_{TM\backslash 0}$ if there exist a $C^0$-Finsler structure $F_{\pi}$ on $\pi^\ast TM$ such that for every $\delta > 0$, there exist a $\beta > 0$ such that 
	\begin{align*}
		F_\pi (x,y,E_\varepsilon(x,y) - E(x,y)) < \delta
	\end{align*}
	for every $(x,y) \in K_{TM\backslash 0}$ and every $\varepsilon \in (0,\beta)$.
\end{definition}

\begin{remark}
	The definition of uniform convergence stated in Definition \ref{definition_uniform_convergence} doesn't depend on the choice of $F_{\pi}$.
\end{remark}


\begin{theorem} 
	\label{convergencia uniforme Chern}
	Let $(M,F)$ be a Finsler manifold and let $F_\varepsilon$ be a mollifier smoothing of $F$. 
	Let $X$ be an arbitrary section into $TM\backslash 0$ and $E$ be an arbitrary section into $\pi^*TM$.
	Then $\left( \nabla_\varepsilon \right)_X E$
	converges uniformly to $\nabla_X E$ on compact subsets $K_{TM \setminus 0}$ of $TM\setminus 0$.
	The statement also holds for the Cartan connection, Hashiguchi connection and Berwald connection.
\end{theorem}
\begin{proof} The proof is almost identical to the proof of Theorem 7 of \cite{Fukuoka-Setti}.
\end{proof}


Now we will study the uniform convergence of the flag curvature of $(M,F_\varepsilon)$ to that of $(M,F)$.
Denote the set of flags on $M$ by $\mathcal F$ and the set of flags on the tangent space $T_xM$ by $\mathcal F_x$.
The flag curvatures of $(M,F)$ and $(M,F_\varepsilon)$ will be denoted by $K$ and $K_\varepsilon$ respectively.


\begin{theorem}
	\label{curvatura flag converge uniforme}
	Let $(M,F)$ be a Finsler manifold and let $F_\varepsilon$ be a mollifier smoothing of $F$.
	Let $K_M$ be a compact subset of $M$ and $\delta >0$.
	Then there exist $\beta > 0$ such that
	\begin{align*}
		\sup_{x \in K_M} \sup_{(y,z) \in \mathcal F_x} \vert K(x,y,z) - K_\varepsilon(x,y,z) \vert < \delta
	\end{align*}
	for every $\varepsilon \in (0, \beta)$.
\end{theorem}
\begin{proof} The proof is almost identical to the proof of Theorem 8 of \cite{Fukuoka-Setti}.
\end{proof}


\subsection{Example}
\label{Example}

If $F$ is a Finsler structure, for each $p \in M$ there exists a coordinate system $\x:U \to \mathbb{R}^n$ of $M$ in $p$ and a positive constant $\gamma_p$ such that $F(x, \cdot)$ is strongly convex with respect to $\sqrt{\frac{\gamma_p}{2}}\Vert \cdot \Vert$. 
It would be natural to ask if a local convexity modulus $\gamma_p$ exists also for strongly convex $C^0$-Finsler structures, whether $\gamma_p$ has its influence on the local convexity modulus of $F_{\varepsilon}$ for a sufficiently small $\varepsilon$. 
We show in the next example that the existence of a local modulus of convexity doesn’t hold for strongly convex $C^0$-Finsler structures in general.

It is natural to ask whether it makes sense to study strongly convex $C^0$-Finsler structures with some kind of ``regularity" on its convexity modulus. The answer is positive in the case of left-invariant structures on Lie groups (see \cite{Fukuoka-Setti2}).
In \cite{Fukuoka-Setti2}, the convexity modulus of $F$ is constant and a common convexity modulus can be defined for $F$ and
its mollifier smoothing. Further cases of regularity will be addressed in a future work.


\begin{example} \rm Consider the Euclidean ball $M=B_{\Vert \cdot \Vert}(0,\frac{1}{2})$ in $\mathbb{R}^2$ endowed with the following radially symmetric $C^0$-Finsler structure $F$:
	\begin{align*}
		F^2(x^1,x^2,y^1,y^2)=\max\left\{(y^1)^2 + (y^2)^2 ,\frac{(y^1)^2}{1-\Vert x\Vert^2} + \Vert x\Vert^2 (y^2)^2 \right\}.
	\end{align*}
	Set $r=\Vert x\Vert$. 
	Direct calculations implies that
	\begin{align*}
		F^2(x^1,x^2,y^1,y^2) = 
		\left\{
		\begin{array}{ccc}
			(y^1)^2 + (y^2)^2 & \text{ if }& \vert y^2 \vert \geq \frac{r}{1-r^2}\vert y^1\vert\\
			\frac{(y^1)^2}{1-\Vert x\Vert^2} + \Vert x\Vert^2 (y^2)^2 & \text{ if }& \vert y^2 \vert \leq \frac{r}{1-r^2}\vert y^1\vert.
		\end{array}
		\right.
	\end{align*} 
	
	Denote by $\left< \cdot, \cdot \right>_r$ the inner product corresponding to the second part of the definition of $F^2$ and denote its norm by $\Vert \cdot \Vert_r$. 
	For a fixed $r>0$, it is not difficult to choose $w,z$ in the region where $F = \Vert \cdot \Vert_r$,  such that $w$ is proportional to $(1,0)$ and $z-w$ is proportional to $(0,1)$. 
	In this case, direct calculations give
	\begin{align}
		\label{align_equacao1_nao_lusc}
		F^2(x, z) & = \Vert z \Vert^2_r = \Vert w \Vert^2_r + 2 \left< w , z-w \right>_r + \Vert z-w \Vert^2_r \nonumber \\
		& = F^2(x,w) + r^2 \Vert z-w \Vert^2, 
	\end{align}
	where the last equality holds because $w \perp z-w$, the fact that $z-w$ is proportional to $(0,1)$ and the definition of $\Vert \cdot \Vert_r$. 
	Now notice that for a fixed $r>0$, $F^2\vert_{T_xM}$ is smooth at $w$ and its differential with respect to the variable $y$ is given by 
	\begin{align*}
		d_y(F^2)_w = \frac{2w^1}{1-\Vert x\Vert^2}dy^1.
	\end{align*}
	Then $d_y(F^2)_w (z-w) = 0$ and \eqref{align_equacao1_nao_lusc} is given by
	\begin{align*}
		F^2(x, z)=F^2(x,w) + d_y (F^2)_w (z-w) + r^2 \Vert z-w \Vert^2
	\end{align*}
	which can be compared to the definition of strong convexity for $F\vert_{T_xM}$.
	But $r>0$ can be taken arbitrarily near zero,  
	what implies that there is no $\gamma_{(0,0)} >0$ such that $F(x, \cdot)$ is strongly convex with respect to $\sqrt{\frac{\gamma_{(0,0)}}{2}} \Vert \cdot \Vert$ for every $x$ in neighborhood of $(0,0)\in M$.
\end{example}


\section{Mollifier smoothing of left-invariant strongly convex $C^0$-Finsler structures on Lie groups}
\label{section Mollifier smoothing of left-invariant strongly convex $C^0$-Finsler structures on Lie Groups}

In the remainder of this work, $G$ is an $n$-dimensional Lie group. We will denote by $e$ and $\mathfrak{g}$ the identity element and the Lie algebra of $G$, respectively. For every $x \in G$ we denote the left translation $x' \in G \mapsto x \cdot x' \in $ by $L_x$.
In Subsection \ref{subsection Mollifier smoothing of left-invariant strongly convex $C^0$-Finsler structures}, we will construct the  left-invariant mollifier smoothing of a left-invariant strongly convex $C^0$-Finsler structure $F:TG \to [0, \infty).$ To do this, we smooth $F_e :=F(e, \cdot): T_eG \cong \mathfrak{g}\cong \mathbb{R}^n \to [0, \infty)$ and use left translations to define the  left-invariant mollifier smoothing $F_\varepsilon$ of $F.$ This smoothing is different from the smoothing defined by \eqref{F epsilon definida em TM}.
In Subsection \ref{subsection Convergence of connections and curvatures on Lie groups}, we will prove that when $F$ is a left-invariant Finsler structure, then objects of the Finsler geometry of $(G, F_\varepsilon)$ such as the fundamental tensor, the Chern connection and the flag curvature converge to the corresponding objects of $(G,F)$.


\subsection{Mollifier smoothing of left-invariant strongly convex $C^0$-Finsler structures}
\label{subsection Mollifier smoothing of left-invariant strongly convex $C^0$-Finsler structures}

Let $F:TG \to [0, \infty)$ be a {\it left-invariant strongly convex $C^0$-Finsler structure}, that is, $F$ is a strongly convex $C^0$-Finsler structure and satisfies the following invariance condition
\begin{align} \label{left-invariant}
	F(x,y)= F(x' \cdot x, (dL_{x'})_{x}(y)),
\end{align}
for every $x, x' \in G$ and $y \in T_{x'}G.$

The vertical smoothing defined in Subsection \ref{section vertical smoothing} holds for $F$. In particular, we can smooth $F_e:\mathbb{R}^n \to [0, \infty)$ instead of $F|_{TU}$ and thus we do not need to consider the neighborhood $U.$
Therefore, with the necessary adaptations, we have that $(\tilde{F}_\varepsilon)_e := \tilde{F}_\varepsilon(e, \cdot):\mathbb{R}^n \to [0, \infty)$ is a Minkowski norm by Proposition \ref{g epsilon radialmente convexo} for every $\varepsilon \in (0, \tau)$, where $\tau$ is defined as in \eqref{define-tau}.

Then, for each $\varepsilon \in (0, \tau)$ we can define $F_\varepsilon: TG \to [0, \infty)$ by
\begin{align} \label{F epsilon definida em TG}
	F_\varepsilon(x, y) = (\tilde{F}_\varepsilon)_e ((dL_{x^{-1}})_x(y)) = \tilde{F}_\varepsilon(e, (dL_{x^{-1}})_x(y)). 
\end{align}
The smoothness of $F_\varepsilon$ on $TG\setminus 0$ is consequence of its left invariance, and $(\tilde{F}_\varepsilon)_e$ being a Minkowski norm with left invariance. Therefore, $F_\varepsilon$ is a left-invariant Finsler structure of $G$ for every $\varepsilon \in (0, \tau)$.


The next theorem was proved in preprint \cite{Fukuoka-Setti2}. Since its proof is available in a preprint and is not entirely straightforward, we have chosen to include it here in order to make the exposition more self-contained and facilitate the reading.  We have tried to remain as faithful as possible to the original proof, making only minor adjustments to the notation and to the references of the results used.

For the next result we will use the diffeomorphism $\varphi:G\times \mathfrak{g} \rightarrow TG $ defined by
\begin{align*} 
	\varphi(x,y)= (x,(dL_x)_e(y)).
\end{align*}
Its inverse $\varphi^{-1}: TG \rightarrow G\times \mathfrak{g}$ is given by
\begin{align*} 
	\varphi^{-1}(x,y)=(x,(dL_{x^{-1}})_x(y)).    
\end{align*}

\begin{theorem} $F_\varepsilon \to F$ uniformly on compact subsets of $TG$ as $\varepsilon$ converges to zero.
\end{theorem}
\begin{proof} Let $K_{TG}$ be a compact subset of $TG.$ Given $\mu>0$ we must find $\tilde{\tau} \in (0, \tau]$ such that
	\begin{align*}
		|F_\varepsilon(x,y) - F(x,y)|  < \mu,
	\end{align*}
	for every $(x,y) \in K_{TG}$ whenever $\varepsilon \in (0, \tilde{\tau}).$
	
	Since $K_{TG} \subset TG$  is compact, then $K_{\mathfrak{g}} = \pi_2(\varphi^{-1}(K_{TG}))$ is a compact subset of $\mathfrak{g}\cong  \mathbb{R}^n,$ where $\pi_2:G\times \mathfrak{g} \rightarrow \mathfrak{g}$ is given by $\pi_2(x,y)=y.$
	
	Note that
	\begin{align*}
		F(x,y) = F_e((dL_{x^{-1}})_x (y)) = F_e(\pi_2(\varphi^{-1}(x,y)))
	\end{align*}
	and
	\begin{align*}
		F_\varepsilon(x,y) = (\tilde{F}_\varepsilon)_e((dL_{x^{-1}})_x (y)) = (\tilde{F}_\varepsilon)_e(\pi_2(\varphi^{-1}(x,y))).
	\end{align*}
	So
	\begin{align*}
		|F_\varepsilon(x,y) - F(x,y)|  = \left|(\tilde{F}_\varepsilon)_e(\pi_2(\varphi^{-1}(x,y))) - F_e(\pi_2(\varphi^{-1}(x,y)))\right|.
	\end{align*}
	
	By Corollary \ref{F_epsilon^til converge uniforme}, for each $\mu>0,$ there exists $\tilde{\tau}\in (0, \tau]$ such that 
	\begin{align*}
		\left|(\tilde{F}_\varepsilon)_e(z) - F_e(z)\right|  < \mu,
	\end{align*}
	for every $z \in K_{\mathfrak{g}}$ whenever $\varepsilon \in (0, \tilde{\tau})$, in particular for $z=\pi_2(\varphi^{-1}(x,y)).$ This proves that $F_\varepsilon$ converges uniformly to $F$ on $K_{TG}$ as $\varepsilon$ converges to zero.
\end{proof}


\begin{definition}
	The one parameter family of applications $F_\varepsilon$ given by \eqref{F epsilon definida em TG} is a {\rm left-invariant mollifier smoothing} of a left-invariant strongly convex $C^0$-Finsler structure $F$.
\end{definition}


\subsection{Convergence of connections and flag curvature on Lie groups}
\label{subsection Convergence of connections and curvatures on Lie groups}

Let $G$ be a Lie group endowed with a left-invariant Finsler structure $F:TG \to [0, \infty)$.
Let $\mathcal{B} = \{e_1, \ldots, e_n \}$ be a basis of $\mathfrak{g}$ and denote the corresponding left-invariant vector fields by $e_i(x) := d(L_x)_e(e_i)$.
Denote the dual basis of $\mathcal{B}$ by $\mathcal{B}^{\ast} = \{\omega^1, \ldots, \omega^n\}$.
Let $\boldsymbol{x} = (x^1, \ldots, x^n):U \rightarrow \mathbb{R}^n$ be a coordinate system on a neighborhood $U$ of $e$ such that $e=(0,\ldots, 0)$. 

The coordinate vector fields and the left-invariant vector fields are related on $U$ by
\begin{align*}
	e_a = w_a^{\ i} \, \frac{\partial}{\partial x^i}, & &  \frac{\partial}{\partial x^i} = v_i^{\ a} \, e_a, & & \\
	\omega^a = v^a_{\ i} \, dx^i, & & dx^i = w^i_{\ a} \, \omega^a, & & 
\end{align*}
where $v$ is the inverse matrix of $w$.

Then,
\begin{align}
	F^2(x,y)  & = F^2(x^\prime \cdot x, d(L_{x^\prime})_x (y)) = F^2(e,d(L_{x^{-1}})_x(y)) \nonumber \\
	& = F_e{}^2 \left(d(L_{x^{-1}})_x \left(y^i \frac{\partial}{\partial x^i}\right)\right) = F_e{}^2\left(d(L_{x^{-1}})_x \left(y^i \, v_i{}^a(x) \, e_a(x) \right)\right) \nonumber \\
	& = F_e{}^2(y^i \, v_i{}^a(x) \, e_a(e)) = F_e{}^2 (y^i \, v_i{}^a(x) \, e_a), \label{align_formula_Fquadrado_algebra}
\end{align}
which is the local representation of the left-invariant $F^2$ on its Lie algebra.
An analogous formula also holds for $F$.
Let us calculate some partial derivatives of $F^2$ with respect to the coordinate system $(x^1, \ldots, x^n, y^1, \ldots, y^n)$.
\begin{align*}
	\frac{\partial F^2}{\partial y^j}(x,y) & = d (F_e{}^2)_{y^i \, v_i{}^a(x) \, e_a}(v_j{}^b(x) \, e_b) =  v_j{}^b(x) \, d (F_e{}^2)_{y^i \, v_i{}^a(x) \, e_a}(e_b) 
\end{align*}
due to \eqref{align_formula_Fquadrado_algebra}. 
Calculating another partial derivative, we have the coefficients of the fundamental tensor:
\begin{align}
	g_{jk(x,y)}=\frac{1}{2} \, \frac{\partial^2 F^2}{\partial y^j\partial y^k}(x,y) & = \frac{1}{2}  \, v_j{}^b(x)  \, v_k{}^c(x)  \, d^2 (F_e{}^2)_{y^i  \, v_i{}^a(x)  \, e_a}(e_b, e_c).  \label{align_formula_fundamental_tensor}
\end{align}

The calculation of the third derivative is analogous.
Considering \eqref{align_formula_fundamental_tensor}, we have
\begin{align}
	C_{jkl(x,y)} & = \frac{1}{2}  \, \frac{\partial}{\partial y^l}g_{jk(x,y)} = \frac{1}{4}  \, \frac{\partial}{ \partial y^l}\left( v_j{}^b(x)  \, v_k{}^c(x)  \, d^2 (F_e{}^2)_{y^i  \, v_i{}^a(x)  \, e_a}(e_b, e_c) \right) \nonumber \\
	& = \frac{1}{4}  \,  v_j{}^b(x)  \, v_k{}^c(x)  \, v_l{}^d(x)  \, d^3 (F_e{}^2)_{y^i  \, v_i{}^a(x)  \, e_a}(e_b, e_c, e_d),  \label{align_Cartan_so-depende-de-Fe}
\end{align}
and we have a corresponding formula for the Cartan tensor $A_{ijk} = F  \, C_{ijk}$ on $\pi^\ast TG$.

Let us calculate $\frac{\partial}{\partial x^l}g_{ij(x,y)}$ considering \eqref{align_formula_fundamental_tensor}, which is a term that appears, for instance, in the formal Christoffel symbols.
Observe that the functions $v_t{}^s$, $s,t=1,\ldots, n$, are the only terms in the right-hand side of \eqref{align_formula_fundamental_tensor} that depend on $x^l$, $l=1,\ldots, n$. 
\begin{align}
	\frac{\partial}{\partial x^l}g_{jk(x,y)} 
	& = \frac{\partial}{\partial x^l}\left(\frac{1}{2} \, v_j{}^b(x) \, v_k{}^c(x) \, d^2 (F_e{}^2)_{y^i \, v_i{}^a(x) \, e_a}(e_b, e_c)\right) \nonumber \\
	& = \frac{1}{2} \, \frac{\partial v_j{}^b}{\partial x^l} \, v_k{}^c \, d^2 (F_e{}^2)_{y^i \, v_i{}^a \, e_a}(e_b, e_c) \nonumber \\
	& + \frac{1}{2} \, v_j{}^b \, \frac{\partial v_k{}^c}{\partial x^l} \, d^2 (F_e{}^2)_{y^i \, v_i{}^a \, e_a}(e_b, e_c) \nonumber \\
	& + \frac{1}{2} \, v_j{}^b \, v_k{}^c \,  y^s \, \frac{\partial v_s{}^d}{\partial x^l} \, d^3(F_e{}^2)_{y^i \, v_i{}^a \, e_a }(e_b, e_c, e_d). \label{align_derivada-gij-em-relacao-a-x}
\end{align}

Proceeding as in \eqref{align_formula_fundamental_tensor}, \eqref{align_Cartan_so-depende-de-Fe} and \eqref{align_derivada-gij-em-relacao-a-x}, it is not difficult to see that partial derivatives of any order of $F^2$ with respect to $x^i$ or $y^j$ can be written as a sum of products of
\begin{itemize} 
	\item $d^m(F_e{}^2)_\mathfrak{y}(e_{a_1}, \ldots, e_{a_m}), \ a_1, \ldots, a_m \in \{1, \ldots, n\}$ and $\mathfrak{y} = y^i \, v_i{}^a \, e_a$; 
	\item $y^i$, $i=1,\ldots, n$; 
	\item $v_j^{\ a}(x)$, $a,j=1, \ldots, n$;
	\item partial derivatives of any order of $v_j^{\ a}(x)$ with respect to variables $x^l$, $l=1,\ldots, n$.
\end{itemize}

The elements $g^{ij}$ of the inverse $g^{-1}$ of the fundamental tensor $g$ can be written as a quotient of sums of products of the elements itemized above due to the formula of $g^{ij}$ in terms of $g_{ij}$.
In the same way, partial derivatives of $g^{ij}$ of any order with respect to $x^i$ or $y^j$ can be written as a quotient of sums of products of elements itemized above.  
Therefore we have the following result due to Equations \eqref{define distinguished}--\eqref{Rjikl abaixa}.

\begin{theorem}
	\label{theorem_geometric_objects}
	The components of the following elements of the Finsler geometry of $(G,F)$ can be written as a quotient of sums of products of
	\begin{itemize} 
		\item $F_e(y^i \, v_i{}^a \, e_a)$;
		\item $d^m(F_e{}^2)_\mathfrak{y}(e_{a_1}, \ldots, e_{a_m}), \ a_1, \ldots, a_m \in \{1, \ldots, n\}$ and $\mathfrak{y} = y^i \, v_i{}^a \, e_a$;
		\item $y^i$, $i=1,\ldots, n$; 
		\item $v_j^{\ a}(x)$, $a,j=1, \ldots, n$;
		\item partial derivatives of any order of $v_j^{\ a}(x)$ with respect to variables $x^l$, $l=1,\ldots, n$.
	\end{itemize}
	\begin{enumerate}
		\item Fundamental tensor: $g_{ij}$.
		\item Inverse tensor: $g^{ij}$.
		\item  Cartan tensor: $A_{ijk}$. 
		\item  Cartan tensor (without normalization): $C_{ijk} = \frac{1}{F} \, A_{ijk}$. 
		\item Formal Christoffel symbols: $\gamma^{i}{}_{jk}.$
		\item Nonlinear connection: $N^i{}_j.$
		\item Chern connection: $\Gamma^i{}_{jk}.$
		\item Horizontal covariant derivative of the Cartan tensor $A$ along the distinguished (horizontal) direction $\hat{\ell} := \ell^i\frac{\delta}{\delta x^i}$: $\dot{A}.$
		\item The following connection forms:
		\begin{itemize}
			\item Cartan connection: 
			\begin{align*}
				\Gamma^i{}_{jk} \, dx^k + C^i_{jk} \, \delta y^k.
			\end{align*}
			\item Hashiguchi connection: 
			\begin{align*}
				\left( \Gamma^i{}_{jk} + \dot{A}^i{}_{jk} \right) dx^k + C^i{}_{jk} \, \delta y^k.
			\end{align*}
			\item Berwald connection:
			\begin{align*}
				\left( \Gamma^i{}_{jk} + \dot{A}^i_{jk} \right) dx^k.
			\end{align*}
		\end{itemize}
		\item $hh$-curvature tensor: $R_j{}^i{}_{kl}.$
	\end{enumerate}
	\item Moreover if we fix $x \in G$ and two transverse vectors $y,z\in T_xG$, the flag curvature $K(y,z)$, depends only on the quotient of sums of products of the aforementioned five items and $z^i$, $i=1,\ldots, n$.
\end{theorem}

\

Let $F_{\varepsilon}$ be a left-invariant mollifier smoothing of $F$.
Now we consider the convergence of geometric objects of Finsler geometry of $(G,F_\varepsilon)$ to the corresponding objects of $(G,F)$ as $\varepsilon \rightarrow 0$.
As in Subsection \ref{secao convergencia conexoes curvatura}, the geometric objects of $(G, F_\varepsilon)$ will be written with a subscript $\varepsilon$.

To verify the convergences, we need the next proposition, which is a version of Corollary 1 of \cite{Fukuoka-Setti} adapted to our situation.

\begin{proposition}
	\label{proposition_vertical_convergence}
	Let $(G,F)$ be a Lie group endowed with a left-invariant Finsler structure. Let $K_{\mathfrak{g}\setminus\{0\}}$ be a compact subset of $\{e\}\times (T_eG\setminus\{0\}) \cong \mathfrak{g}\setminus\{0\}$ and $\alpha = (\alpha_1,\ldots, \alpha_n)$ be a multi-index. Then,
	\begin{align*}
		D^\alpha \left( (\tilde{F}_\varepsilon)_e{}^2 \right) = \frac{\partial^{|\alpha|} (\tilde{F}_\varepsilon)_e{}^2}{\partial(y^1)^{\alpha_{1}} \cdots \partial(y^n)^{\alpha_{n}}}
	\end{align*}
	converges uniformly to
	\begin{align*}
		D^\alpha (F_e{}^2) = \frac{\partial^{|\alpha|} F_e{}^2}{\partial(y^1)^{\alpha_{1}} \cdots \partial(y^n)^{\alpha_{n}}}
	\end{align*}
	on $K_{\mathfrak{g}\setminus\{0\}}$.
\end{proposition}

Considering $(g_\varepsilon)_{ij}$, $\frac{\partial}{\partial y^k}(g_\varepsilon)_{ij}$ and $\frac{\partial}{\partial x^l}(g_\varepsilon)_{ij}$ in \eqref{align_formula_fundamental_tensor}, \eqref{align_Cartan_so-depende-de-Fe} and \eqref{align_derivada-gij-em-relacao-a-x} respectively, the only terms of the right-hand side of these equations that depend on $\varepsilon$ are the differentials of $(\tilde F_\varepsilon)_e{}^2$.
The terms $y^i$, $v_t^{\ s}$ and their derivatives don't depend on $\varepsilon$ (if we want to study the flag curvature, $z^i$ don't depend on $\varepsilon$ either). 
But we know that $d^m((\tilde{F}_\varepsilon)_e{}^2)_\mathfrak{y}(e_{a_1}, \ldots, e_{a_m}), a_1, \ldots, a_m \in \{1, \ldots, n\}$, converge uniformly to $d^m(F_e{}^2)_\mathfrak{y}(e_{a_1}, \ldots, e_{a_m})$ when $\varepsilon \rightarrow 0$ on compact subsets of $\mathfrak{g}\backslash \{0\}$ due to Proposition \ref{proposition_vertical_convergence}.
Therefore $(g_\varepsilon)_{ij}$, $\frac{\partial}{\partial y^k}(g_\varepsilon)_{ij}$ and $\frac{\partial}{\partial x^l}(g_\varepsilon)_{ij}$ converge uniformly to $g_{ij}$, $\frac{\partial}{\partial y^k}g_{ij}$ and $\frac{\partial}{\partial x^l}g_{ij}$ on compact subsets of $\mathfrak{g}\backslash \{0\}$, respectively.
We also have that $(\tilde{F}_\varepsilon)_e$ converges uniformly to $F_e$ on compact subsets of $\mathfrak{g}\backslash \{0\}$.

Notice that the same idea is applicable for every object defined in Theorem \ref{theorem_geometric_objects}, and we provide this result for future reference.  

\begin{proposition}
	\label{proposition_convergence_objects_lie_groups}
	Let $(G,F)$ be a Lie group endowed with a left-invariant Finsler structure, $\boldsymbol{x}: U \rightarrow \mathbb{R}^n$ be a coordinate system in a neighborhood of $e \in G$, and $(\boldsymbol{x}, \boldsymbol{y})$ be the corresponding natural coordinate system on $TU$. Consider the geometric objects defined in Theorem \ref{theorem_geometric_objects}.
	Denote the components of these objects by $T^{i_1 i_2 \ldots i_m}_{j_1 j_2 \ldots j_s}$ and the corresponding objects of $(G,F_\varepsilon)$ by $(T_\varepsilon)^{i_1 i_2 \ldots i_m}_{j_1 j_2 \ldots j_s}$.
	Then  $(T_\varepsilon)^{i_1 i_2 \ldots i_m}_{j_1 j_2 \ldots j_s}$ converges uniformly to  $T^{i_1 i_2 \ldots i_m}_{j_1 j_2 \ldots j_s}$ on compact subsets of  $TU\backslash 0$.
\end{proposition}

Versions of Theorems \ref{convergencia uniforme Chern} and \ref{curvatura flag converge uniforme} for Lie groups are straightforward from
Proposition \ref{proposition_convergence_objects_lie_groups}
and Theorems 7 and 8 of \cite{Fukuoka-Setti}.
\begin{theorem} 
	\label{convergencia uniforme Chern grupos de Lie}
	Let $(G,F)$ be a Lie group with a left-invariant Finsler structure and $F_\varepsilon$ be a left-invariant mollifier smoothing of $F$. 
	Let $X$ be an arbitrary section into $TG\backslash 0$ and $E$ be an arbitrary section into $\pi^*TG$.
	Then $\left( \nabla_\varepsilon \right)_X E$
	converges uniformly to $\nabla_X E$ on compact subsets $K_{TG \setminus 0}$ of $TG\setminus 0$.
	The statement also holds for the Cartan connection, Hashiguchi connection and Berwald connection.
\end{theorem}

\begin{theorem}
	\label{curvatura flag converge uniforme grupos de Lie}
	Let $(G,F)$ be a Lie group with a left-invariant Finsler structure and $F_\varepsilon$ be a left-invariant mollifier smoothing of $F$.  
	Let $K_G$ be a compact subset of $G$ and $\delta >0$.
	Then there exist $\beta > 0$ such that
	\begin{align*}
		\sup_{x \in K_G} \sup_{(y,z) \in \mathcal F_x} \vert K(x,y,z) - K_\varepsilon(x,y,z) \vert < \delta
	\end{align*}
	for every $\varepsilon \in (0, \beta)$.
\end{theorem}

\end{document}